\documentclass[12pt]{amsart}
\usepackage{graphicx}
\usepackage{amsmath}
\usepackage{amsthm}
\usepackage{amsfonts}
\usepackage{amssymb}
\usepackage[usenames,dvipsnames]{color}
\usepackage{verbatim}

\theoremstyle{plain}
\newtheorem{thm}{Theorem}[section]
\newtheorem{lemma}[thm]{Lemma}
\newtheorem{cor}[thm]{Corollary}
\newtheorem{prop}[thm]{Proposition}

\newtheorem*{thm*}{Theorem}

\theoremstyle{remark}
\newtheorem{rmk}[thm]{Remark}

\newtheorem{claim}[thm]{Claim}

\theoremstyle{definition}
\newtheorem{defn}[thm]{Definition}

\numberwithin{equation}{section}

\def\N{\mathbb{N}}

\def\R{\mathbb{R}}

\def\1{\mathbf{1}}

\DeclareMathOperator{\diam}{diam}

\newenvironment{PfofThm1}[1]
{\par\vskip2\parsep\noindent{\sc Proof of Theorem \ \ref{Thm:SubshiftEntropies}. }}{{\hfill
$\Box$}
\par\vskip2\parsep}

\newenvironment{PfofThm2}[1]
{\par\vskip2\parsep\noindent{\sc Proof of Theorem \ \ref{Thm:RelativeSubshiftEntropies}. }}{{\hfill
$\Box$}
\par\vskip2\parsep}

\newenvironment{PfofThm3}[1]
{\par\vskip2\parsep\noindent{\sc Proof of Theorem \ \ref{Thm:ZeroDimSystems}. }}{{\hfill
$\Box$}
\par\vskip2\parsep}

\title{Ubiquity of entropies of intermediate factors}

\begin{document}
%

\author{Kevin McGoff}
\address{Kevin McGoff\\
Department of Mathematics and Statistics\\
University of North Carolina at Charlotte \\
9201 University City Blvd.\\
Charlotte, NC 28223}
\email{kmcgoff1@uncc.edu}
\urladdr{https://clas-math.uncc.edu/kevin-mcgoff/}

\author{Ronnie Pavlov}
\address{Ronnie Pavlov\\
Department of Mathematics\\
University of Denver\\
2390 S. York St.\\
Denver, CO 80208}
\email{rpavlov@du.edu}
\urladdr{http://www.math.du.edu/$\sim$rpavlov/}

\thanks{The first author gratefully acknowledges support from National Science Foundation grants DMS-1613261 and DMS-1847144. The second author gratefully acknowledges the support of a Simons Foundation Collaboration Grant.}
\keywords{Symbolic dynamics, amenable, topological entropy}
\renewcommand{\subjclassname}{MSC 2010}
\subjclass[2010]{Primary: 37B10; Secondary: 37B50, 37B40}

\begin{abstract}
We consider topological dynamical systems $(X,T)$, where $X$ is a compact metrizable space and $T$ denotes an action of a countable amenable group $G$ on $X$ by homeomorphisms. 
For two such systems $(X,T)$ and $(Y,S)$ and a factor map $\pi : X \rightarrow Y$, an intermediate factor is a topological dynamical system $(Z,R)$ for which $\pi$ can be written as a composition of factor maps $\psi : X \rightarrow Z$ and $\varphi : Z \rightarrow Y$. 
In this paper we show that for any countable amenable group $G$, for any $G$-subshifts $(X,T)$ and $(Y,S)$, and for any factor map $ \pi :X \rightarrow Y$, the set of entropies of intermediate subshift factors is dense in the interval $[h(Y,S), h(X,T)]$. 
As a corollary, we also prove that if $(X,T)$ and $(Y,S)$ are zero-dimensional $G$-systems, then the set of entropies of intermediate zero-dimensional factors is equal to the interval $[h(Y,S), h(X,T)]$.
Our proofs rely on a generalized Marker Lemma that may be of independent interest.
\end{abstract}

\maketitle



\section{Introduction}
\label{Sect:Intro}

In this work, we continue a line of research, initiated by Shub and Weiss \cite{shubweiss}, concerning the following seemingly basic question: given a topological dynamical system, what can be said about the topological entropies of its factors? For the purposes of this paper, we consider a topological dynamical system to be a pair $(X,T)$, where $X$ is a compact metrizable space and $T$ is an action of a countable amenable group $G$ on $X$ by homeomorphisms. Additionally, a factor of such a system $(X,T)$ is another system $(Y,S)$ for which there exists a continuous surjection $\pi : X \to Y$ that commutes with the actions of $S$ and $T$. 
Some general results about the entropies of factors are available \cite{linden,lindenstrauss1999,shubweiss}, but the question above has not been completely resolved.
For instance, it is still not known whether every system with positive, finite entropy must have a nontrivial factor with strictly smaller entropy. 

Under certain hypotheses, there are some existing results about the entropies of factors. For instance, Lindenstrauss proved \cite{linden} that if $(X,T)$ is a finite-dimensional $\mathbb{Z}$-system, then 
the set of entropies of factors of $(X, T)$ is the entire interval $[0, h(X,T)]$. Furthermore, Lindenstrauss provided examples showing that this is not necessarily the case if $X$ has infinite dimension.

Lindenstrauss also proved what we call a `relative version' of his result, which concerns the entropies of intermediate factors. Given a countable amenable group $G$, $G$-systems $(X,T)$ and $(Y,S)$, and a factor map $\pi : X \rightarrow Y$, an intermediate factor is a $G$-system $(Z,R)$ for which $\pi$ can be written as a composition of factor maps $\psi : X \rightarrow Z$ and $\varphi : Z \rightarrow Y$. 
With this definition, Lindenstrauss showed that if $\pi$ is a factor map from a finite dimensional $\mathbb{Z}$-system
$(X, T)$ onto $(Y, S)$, then the set of entropies of intermediate factors is the entire interval $[h(Y,S), h(X,T)]$ \cite{linden}. (The relative version is of course more general than the original result concerning entropies of factors of a single system, since the original result may be obtained by taking $(Y, S)$ to be the trivial factor in the relative version.)
Using the notion of mean dimension \cite{lindenstrauss2000}, Lindenstrauss also generalized these results to extensions of minimal $\mathbb{Z}$-systems with zero mean dimension \cite{lindenstrauss1999}.

Here we examine the question in the setting of subshifts, or symbolic dynamical systems. Of course subshifts are zero-dimensional, and therefore previous results about zero-dimensional systems can be applied to subshifts. However, we are interested in the finer question of what can be said about the entropies of intermediate factors, where the intermediate systems must come from a restricted class (such as subshifts or zero-dimensional systems). In this direction, for $\mathbb{Z}$-subshifts, it was shown in \cite{shubweiss} that any system with positive entropy must have nontrivial subshift factors of strictly smaller entropy (and in fact that they can be taken arbitrarily close to $h(X,T)$). In addition, in \cite{hongetal}, in the case where $(X, T)$ is a sofic $\mathbb{Z}$-subshift, it was shown that the set of entropies of subshift factors is dense in the interval $[0, h(X,T)]$, which is in some sense the most that can be hoped for, since a subshift has only countably many subshift factors by the Curtis-Lyndon-Hedlund theorem. A corresponding relative result was also established when $(X, T)$ and $(Y, S)$ are both sofic $\mathbb{Z}$-subshifts. 

In this work, we show that the previously mentioned result of \cite{hongetal} holds even when $(X, T)$ is an arbitrary (not necessarily sofic) subshift. Furthermore, we generalize these results from $\mathbb{Z}$-subshifts to $G$-subshifts, where $G$ is an arbitrary countable amenable group. Note that the computation/realization of entropies is much more difficult in the setting of general countable amenable groups. For instance, realization of arbitrary entropies for $G$-subshifts has only recently been addressed for general classes of countable amenable groups \cite{frischetal,kwietniak}. 
In order to establish our main results at this level of generality, we first prove a generalized Marker Lemma for countable groups.

Let us now state our main results. 
In the following two results, $G$ denotes any countable amenable group, $(X,T)$ and $(Y,T)$ represent $G$-subshifts,
$\mathcal{H}_{sub}(X)$ denotes the set of numbers $r \in \mathbb{R}$ such that there exists a $G$-subshift $(Z,R)$ so that $(X,T)$ factors onto $(Z,R)$ and $h(Z,R) = r$, and
$\mathcal{H}_{sub}^{\pi}(X,Y)$ denotes the set of numbers $r \in \R$ such that there exists a $G$-subshift $(Z,R)$ with $h(Z,R) = r$ and factor maps $\varphi : X \to Z$ and $\psi : Z \to Y$ such that $\pi = \psi \circ \varphi$.

\begin{thm} \label{Thm:SubshiftEntropies}
Let $G$ be a countable amenable group, and let $(X,T)$ be a $G$-subshift. Then $\mathcal{H}_{sub}(X)$ is dense in the interval $[0,h(X,T)]$.
\end{thm}

The main technique in the proof of Theorem~\ref{Thm:SubshiftEntropies} is a general version of the classical Krieger Marker Lemma 
(see \cite{Boyle1983}, \cite{Krieger1982}), which may be of independent interest; see Section~\ref{marker}. We also establish the following relative version.

\begin{thm} \label{Thm:RelativeSubshiftEntropies}
Let $G$ be a countable amenable group. Let $(X,T)$ and $(Y,S)$ be $G$-subshifts with a factor map $\pi : X \to Y$. Then $\mathcal{H}_{sub}^{\pi}(X,Y)$ is dense in the interval $[h(Y,S),h(X,T)]$. 
\end{thm}

From these results, we are actually able to establish a similar result about general zero-dimensional systems as well.

In the following result, for zero-dimensional $G$-systems $(X,T)$ and $(Y,S)$ and a factor map $\pi: X \rightarrow Y$, we let $\mathcal{H}_{0}^{\pi}(X,Y)$ denote the set of numbers $r \in \R$ such that there exists a zero-dimensional $G$-system $(Z,R)$ with $h(Z,R) = r$ and factor maps $\varphi : X \to Z$ and $\psi : Z \to Y$ such that $\pi = \psi \circ \varphi$.






\begin{thm} \label{Thm:ZeroDimSystems}
Let $G$ be a countable amenable group, and let $\pi: (X,T) \rightarrow (Y,S)$ be a factor map between zero-dimensional $G$-systems. 
Then $\mathcal{H}_{0}^{\pi}(X,Y) = [h(Y,S), h(X,T)]$.
\end{thm}

We note that this result provides more information about the aforementioned result of \cite{linden} when $(X,T)$ and $(Y,S)$ are zero-dimensional: in this case, not only do we know that every entropy in $[h(Y,S), h(X,T)]$ can be achieved by factors, but we now know that those factors can always be chosen to be zero-dimensional.

The rest of the paper is organized as follows.
Section~\ref{defs} contains necessary background and notation, Sections~\ref{marker} and \ref{density} contain some preliminary results required for later proofs, and Sections~\ref{mainprf}, \ref{relprf}, 
and \ref{zeroprf} contain the proofs of 
Theorems~\ref{Thm:SubshiftEntropies}-\ref{Thm:ZeroDimSystems}.

\section*{Acknowledgments}

The authors would like to thank Uijin Jung for introducing these questions to them and for many useful conversations. Additionally, the authors wish to thank Mike Boyle for asking them about zero-dimensional factors, thus motivating Theorem~\ref{Thm:ZeroDimSystems}. The authors also wish to thank Tomasz Downarowicz for helpful conversations regarding zero-dimensional factors and topological joinings. 

\section{Background and notation}\label{defs}

%
%

\subsection{Countable amenable groups}

Let $G$ be a countable amenable group. For sets $A,K \subset G$, we let $A K = \{ a k : a \in A, k \in K\}$ and $A \, \Delta \, K = (A \setminus K) \cup (K \setminus A)$. A sequence $\{F_n\}_{n=1}^{\infty}$ of finite subsets of $G$ is called a F{\o}lner sequence if for each finite set $K \subset G$, we have
\begin{equation*}
\lim_n \frac{ | K F_n  \, \Delta \, F_n|}{ |F_n|} = 0.
\end{equation*}
The existence of a F{\o}lner sequence is equivalent to the amenability of the group $G$. Let $\{F_n\}_{n=1}^{\infty}$ be a F{\o}lner sequence, which we fix here and use throughout the paper.

\begin{defn}
Suppose $U$ and $K$ are non-empty finite subsets of $G$ and $\delta >0$. We say that $U$ is \textit{$(K,\delta)$-invariant} if
\begin{equation*}
\frac{| KU \, \triangle \, U |}{|U|} < \delta.
\end{equation*}
\end{defn}
Note that the definition of F{\o}lner sequence yields that for any finite set $K \subset G$ and any $\delta >0$, the set $F_n$ is $(K,\delta)$-invariant for all large enough $n$.

\begin{defn}
Suppose $\{A_1,\dots,A_r\}$ is a collection of finite sets of $G$. We say that the collection is \textit{$\delta$-disjoint} if there exist subsets $\{B_1,\dots,B_r\}$ such that
\begin{enumerate}
\item $B_i \subset A_i$ for all $i$,
\item for $i \neq j$, we have $B_i \cap B_j = \varnothing$, and
\item $|B_i|/|A_i| > 1-\delta$ for all $i$.
\end{enumerate}
\end{defn}

\begin{defn}
Suppose $\{A_1,\dots,A_r\}$ is a collection of finite sets of $G$, and let $B \subset G$. We say that the collection \textit{$\alpha$-covers} $B$ if
\begin{equation*}
 \bigl| B \cap \bigl( \bigcup_i A_i \bigr) \bigr| \geq \alpha |B|.
\end{equation*}
\end{defn}

\begin{defn}{\rm (\cite{OW1987})}
For $\delta >0$, a collection of finite sets $\{T_1,\dots,T_N\}$ is said to \textit{$\delta$-quasitile} a group $G$ (or to be a \textit{set of $\delta$-quasitiles for $G$}) if $\{e\} \in T_1 \subset \dots \subset T_N$ and for any finite set $D \subset G$, there are finite sets $C_i \subset G$,  for $1 \leq i \leq N$, such that 
\begin{enumerate}
\item for fixed $i$, the collection $\{T_i c : c \in C_i \}$ is $\delta$-disjoint
\item for $i \neq j$, $T_i C_i \cap T_j C_j = \varnothing$, and
\item the collection $\{T_1 C_1, \dots, T_N C_N \}$ $(1-\delta)$-covers $D$.
\end{enumerate}
\end{defn}

The fundamental result of Ornstein and Weiss \cite[Theorem 6]{OW1987} states that for any $\delta >0$, there exists $N$ such that for any finite set $K$ and any $\delta' >0$, there exist finite sets $\{T_1,\dots,T_N\}$ that are $(K,\delta')$-invariant and that $\delta$-quasi-tile $G$. We remark that by making $\delta' >0$ small enough and $K$ large enough, one may force $|T_1| = \min_i |T_i|$ to be arbitrarily large. 

For a collection of quasitiles $\{T_1,\dots,T_N\}$, we refer to any collection $\{C_1,\dots,C_N\}$ satisfying (1)-(3) above as a collection of \textit{center sets} corresponding to $D$. If $\{C_1,\dots,C_N\}$ is a collection of center sets for $D$, then let $C'_i = \{ c \in C_i : T_i c \cap D \neq \varnothing\}$, and observe that the collection $\{C'_1,\dots,C'_N\}$ is again a collection of center sets for $D$. Thus, for a collection of quasitiles $\{T_1,\dots,T_N\}$ and finite $D \subset G$, there exists a collection of center sets $\{C_1,\dots,C_N\}$, and we may assume without loss of generality that  for all $c \in C_i$, we have $T_i c \cap D \neq \varnothing$.

\begin{lemma} \label{Lemma:Coverings}
Let $\delta \in (0,1)$.
Suppose that $\{S_1,\dots,S_N\}$ is a collection of finite sets that $\delta$-quasitiles $G$. Let $m = |S_1| = \min \{ |S_1|, \dots, |S_N|\}$, and $S = \bigcup_i S_i = S_N$.
Then for all large enough $n$, there exists $C = C(n) \subset G$ such that $\{S c : c \in C\}$ $(1-\delta)$-covers $F_n$ and 
\begin{equation*}
|C| \leq \frac{(1+\delta) |F_n|}{(1-\delta) m }.
\end{equation*}
\end{lemma}
\begin{proof}
Since $\{F_n\}_{n=1}^{\infty}$ is a F{\o}lner sequence, for all large enough $n$, we have
\begin{equation} \label{Eqn:Richter}
\frac{| S S^{-1} F_n \,  \triangle \, F_n|}{|F_n|} < \delta.
\end{equation}
Fix $n$ satisfying this condition.
Since $\{S_1,\dots,S_N\}$ is a collection of $\delta$-quasitiles, there exists a collection of center sets corresponding to the set $F_n$ with the additional property that for each $c \in C_i$, we have $S_i c \cap F_n \neq \varnothing$. 

Let $C = \bigcup_i C_i$. Then 
\begin{equation*}
A = \bigcup_{i} S_i C_i \subset S C,
\end{equation*}
and since $A$ $(1-\delta)$-covers $F_n$, we see that $S C$ $(1-\delta)$-covers $F_n$. 

Now let us estimate $|C|$. Let $C_i = \{c^i_1,\dots,c^i_{M_i}\}$, and let $A^i_k = S_i c^i_k$, for $k = 1, \dots, M_i$. Then the collection $\{A^i_k : 1 \leq i \leq N, 1 \leq k \leq M_i \}$ is $\delta$-disjoint (by the quasitile properties (1) and (2)). Then there exist sets $B^i_k \subset A^i_k$ such that $|B^i_k| \geq |A^i_k| (1-\delta)$ and if $B^i_k \cap B^j_{\ell} \neq \varnothing$, then $i = j$ and $k = \ell$. 
Then we have
\begin{align*}
\Biggl| \bigcup_{i=1}^N \bigcup_{k = 1}^{M_i} B^i_k \Biggr| & = \sum_{i=1}^N \sum_{k=1}^{M_i} |B^i_k| \\
& \geq  \sum_{i=1}^N \sum_{k=1}^{M_i} (1-\delta) |A^i_k| \\
& = (1-\delta) \sum_{i=1}^N \sum_{k=1}^{M_i} |S_i| \\
& = (1-\delta) \sum_{i = 1}^N |S_i| \cdot |C_i| \\
& \geq (1-\delta) m \sum_{i=1}^N |C_i| \\
& \geq (1-\delta) m |C|.
\end{align*}
Fix $1\leq i \leq N$ and $1 \leq k \leq M_i$. 
Recall that $A^i_k \cap F_n = S_i c^i_k \cap F_n \neq \varnothing$. Let $g \in S_i c^i_k \cap F_n$. Then there exists $s \in S_i$ such that $g = s c^i_k$, which implies that $c^i_k = s^{-1} g \in S^{-1} F_n$, and hence $S c^i_k \subset S S^{-1} F_n$. Then we see that
\begin{align*}
\Biggl| \bigcup_{i=1}^N \bigcup_{k = 1}^{M_i} B^i_k \Biggr| & \leq \Biggl| \bigcup_{i=1}^N \bigcup_{k = 1}^{M_i} A^i_k \Biggr| \\
& \leq | S S^{-1} F_n| \\
& \leq (1+\delta) |F_n|,
\end{align*}
where we have used (\ref{Eqn:Richter}).
Combining the above inequalities gives
\begin{align*}
 (1-\delta) m |C| \leq \Biggl| \bigcup_{i=1}^N \bigcup_{k = 1}^{M_i} B^i_k \Biggr|  \leq (1+\delta) |F_n|,
\end{align*}
from which we conclude that
\begin{equation*}
|C| \leq \frac{(1+\delta) |F_n|}{(1-\delta) m },
\end{equation*}
as desired.
\end{proof}

\subsection{Topological and symbolic dynamics}

We present some basic definitions from topological and symbolic dynamics; for a more thorough introduction to symbolic dynamics, see \cite{LindMarcus}.

\begin{defn}
Let $G$ be a countable amenable group. A $G$-\textit{system} is a pair $(X,T)$, where $X$ is a compact metrizable space and $T = (T^g)_{g \in G}$ is an action of $G$ on $X$ by homeomorphisms. 
\end{defn}

\begin{defn}
A \textit{factor map} from a $G$-system $(X,T)$ to a $G$-system $(Y,S)$ is a surjective function $\varphi : X \to Y$ that is continuous and commutes with the actions of $S$ and $T$. A \textit{conjugacy} is a factor map that is also injective. When $\varphi : X \rightarrow Y$ is a factor map, we refer to $(Y,S)$ as a factor of $(X,T)$.
\end{defn}


\begin{defn}
For any finite alphabet $\mathcal{A}$ and countable amenable group $G$, the \textit{full $G$-shift} over $\mathcal{A}$ is the set $\mathcal{A}^{G}$, which is viewed as a compact topological space with the (discrete) product topology.
\end{defn}

\begin{defn}
A \textit{pattern} over $\mathcal{A}$ is any function $w$ from a finite set $S \subset G$ to $\mathcal{A}$, in which case $S$ is called the \textit{shape} of $w$.
\end{defn}

\begin{defn}
The \textit{shift action} $\sigma = (\sigma^g)_{g \in G}$ is the action of $G$ on $\mathcal{A}^G$ by automorphisms defined by $(\sigma^g (x))(h) = x(hg)$ for all $g,h \in G$. 
\end{defn}

\begin{defn}
A \textit{$G$-subshift} is a closed subset of a full shift $\mathcal{A}^{G}$ that is invariant under $\sigma$.
\end{defn}

Any subshift $X$ is a compact space with the induced topology from $\mathcal{A}^G$, and so $(X,\sigma|_X)$ is a topological $G$-dynamical system. In fact, this topology on $X$ is also generated by the ultra-metric given by
\begin{equation*}
d(x,y) = 2^{- \min\{ n \, : \, x_{g_n} \neq y_{g_n} \}},
\end{equation*}
where $G = \{g_n\}_{n=0}^{\infty}$ is an enumeration of $G$. 

\begin{defn}
For any $G$-subshift $X$ and finite $S \subset G$, the \textit{$S$-language} of $X$, denoted by $\mathcal{L}_S(X)$, is the set of all patterns $w$ with shape $S$ 
which appear as subpatterns of points of $X$, i.e., for which there exists $x \in X$ and $g \in G$ with $x(Sg) = w$.
\end{defn}

\begin{defn}
For any $G$-subshift $X$ and $w \in \mathcal{L}_S(X)$, the \textit{cylinder set} $[w]$ is the set of all $x \in X$ with $x(S) = w$.  
\end{defn}

All factor maps between subshifts have a simple combinatorial description.

\begin{defn}\label{SBC}
Given alphabets $\mathcal{A}$ and $\mathcal{B}$, a finite set $S \subset G$, and a function $f: \mathcal{A}^S \rightarrow \mathcal{B}$, the \textit{sliding block code} induced by $S$ and $f$ is the map $\varphi: \mathcal{A}^G \rightarrow \mathcal{B}^G$ defined by 
\[
(\varphi(x))(g) = f(x(Sg))
\]
for all $x \in \mathcal{A}^G$ and $g \in G$. A \textit{$1$-block code} is a sliding block code when $S = \{e\}$.
\end{defn}

If $X$ and $Y$ are subshifts, then for every factor map $\pi : X \rightarrow Y$, there is a sliding block code $\varphi$ such that $\pi = \varphi|_X$; this is the classical Curtis-Lyndon-Hedlund theorem when $G = \mathbb{Z}$, and the proof is essentially the same for general $G$. Even more is true: up to conjugacy, every factor map can be written as a $1$-block code, i.e.,
if $\varphi: X \rightarrow Y$ is a factor map, then there exists a conjugacy $\psi$ on $X$ and a $1$-block code $\rho$ on $\psi(X)$ so that
$\varphi = \rho \circ \psi$.

\subsection{Entropy}

When $G$ is a countable amenable group, one may associate to each $G$-system $(X,T)$ an extended real number called the \textit{topological entropy} of $(X,T)$. 
The entropy of a system quantifies its complexity and serves as an important conjugacy invariant in the study of topological dynamical systems.
The entropy theory for countable amenable group actions is well-established \cite{kerrli,OW1987}. Here we state the main entropy-related definitions and properties that we require elsewhere in the paper. We begin with topological entropy and topological conditional entropy.

\begin{defn}
Suppose $\pi : X \to Y$ is a factor map between $G$-systems $(X,T)$ and $(Y,S)$. Let $\mathcal{U}$ be a finite open cover of $X$. For any subset $K \subset X$, let $N(\mathcal{U} \mid K)$ be the minimal cardinality of any subcover of $\mathcal{U}$ that covers $K$. Define
\begin{equation*}
h( T, \mathcal{U} \mid S) = \lim_n \frac{1}{|F_n|} \log \, \sup_{y \in Y} N( \mathcal{U}_{F_n} \mid \pi^{-1}(y) ),
\end{equation*}
where $\mathcal{U}_{F_n} = \bigvee_{g \in F_n} T^{g^{-1}} \mathcal{U}$ and $\{F_n\}$ is a F\o lner sequence.
Then the \textit{conditional entropy of $(X,T)$ with respect to $(Y,S)$} is given by 
\begin{equation*}
h(T \mid S) = \sup_{\mathcal{U}} h( T, \mathcal{U} \mid S),
\end{equation*}
where the supremum is taken over all finite open covers $\mathcal{U}$. 
When $(Y,S)$ is the trivial (one-point) system, we call $h(T \mid S)$ the \textit{topological entropy of $(X,T)$}, and denote it simply by $h(T)$. 
\end{defn}

In Section \ref{Sect:Intro}, we use the notation $h(X,T)$ to denote the topological entropy, but throughout the rest of the paper we use the notation $h(T)$ (noting that the dependence on $X$ is implicit, as it is the domain of the map $T$).
We remark that the above limit always exists and is independent of the choice of F\o lner sequence. Furthermore, if $(X,T)$ is a subshift and $\mathcal{U}$ is the clopen partition of $X$ according to the symbol at the identity element of $G$, then $\mathcal{U}$ is generating, and therefore $h(T \mid S) = h(T, \mathcal{U} \mid S)$. Lastly, we note that $h(T \mid S) \leq h(T)$. 

%
%
%

It is well-known that for any countable amenable group $G$, factor maps cannot increase topological entropy, i.e. if $\varphi$ is a factor map from $(X,T)$ to $(Y,S)$, then $h(T) \geq h(S)$. This observation motivates the main question of this paper, namely whether a large set of entropies is achieved for intermediate systems `between $X$ and $Y$.'

In the other direction, it is known that the decrease in topological entropy under a factor map is bounded above by the topological conditional entropy. Results of this type have been established by Bowen \cite{bowenendo} and by Ledrappier and Walters \cite{Ledrappier1977} for $\mathbb{Z}$ actions. We require the following corollary of results from \cite{Yan2015}.
\begin{thm}[\cite{Yan2015}]\label{bowenfiber}
For any factor map $\pi: X \rightarrow Y$ between $G$-systems $(X, T)$ and $(Y, S)$, 
\[
h(T) \leq h(S) + h(T \mid S).
\]
\end{thm}

\subsection{Topological joinings} \label{Sect:TopJoinings} 

Suppose $(X,T)$ and $(Y,S)$ are $G$-systems. A \textit{topological joining} of these systems is a $G$-system $(Z,R)$ such that $Z \subset X \times Y$, the action $R$ is defined by $R^g(x,y) = (T^g(x), S^g(y))$, and the natural projections of $Z$ onto each its components gives factor maps onto $(X,T)$ and $(Y,S)$, respectively. For any such topological joining, we note that $h(R) \leq h(S) + h(T)$ (as $(Z,R)$ is a subsystem of the direct product system, which has entropy $h(S) + h(T)$).
 
Here we describe a specific construction of topological joinings that we use in some of our proofs. Suppose that $(X,T)$, $(Y,S)$, and $(Z,R)$ are three $G$-systems, and $\pi : X \to Y$ and $\varphi : X \to Z$ are factor maps. Then we define $J(\pi,\varphi)$, the \textit{topological joining of $\pi$ and $\phi$}, to be another $G$-system $(W,Q)$, defined as follows. The space $W$ is
\begin{equation*}
W = \bigl\{ (\pi(x), \varphi(x)) : x \in X \bigr\},
\end{equation*}
and the action $Q$ is given by $Q^g(y,z) = (S^g(y), R^g(z))$. 
Note that there are three factor maps associated to this construction:
\begin{itemize}
\item $\psi : X \to W$, where $\psi(x) = (\pi(x), \varphi(x))$, 
\item $\pi_1 : W \to Y$, where $\pi_1(y,z) = y$, and 
\item $\pi_2 : W \to Z$, where $\pi_2(y,z) = z$. 
\end{itemize}

We state the following fact for future use. 
\begin{lemma} \label{Lemma:SpecialFlower}
Suppose $(X_0,T_0)$, $(X_1,T_1)$, and $(Y,S)$ are $G$-systems and $\pi : X_0 \to X_1$ is a factor map. If $(W_i,R_i)$ is a topological joining of $(X_i,T_i)$ and $(Y,S)$ for $i=0,1$ and $\varphi : W_0 \to W_1$ is the factor map defined by $\varphi(x,y) = (\pi(x),y)$, then $h(R_0 \mid R_1) \leq h(T_0 \mid T_1)$. 
\end{lemma}
\begin{proof}
Let $\mathcal{U}$ be a finite open cover of $W_0$. Since the projection maps from $W_0$ onto $X_0$ and $Y$ are open maps, the projections of $\mathcal{U}$ onto these factors yield open covers $\mathcal{U}^0$ of $X_0$ and $\mathcal{U}^Y$ of $Y$. Also, for any $(x,y) \in W_1$, note that $\varphi^{-1}(x,y) = \pi^{-1}(x) \times \{y\}$. Hence, for any $n$, we have $N( \mathcal{U}_{F_n} \mid \varphi^{-1}(x,y) ) \leq N( \mathcal{U}^0_{F_n} \mid \pi^{-1}(x) ) \cdot N( \mathcal{U}^Y_{F_n} \mid \{y\}) = N( \mathcal{U}^0_{F_n} \mid \pi^{-1}(x) )$. Then the desired inequality follows from the definition of topological conditional entropy.
\end{proof}

\subsection{Zero-dimensional systems} 
A $G$-system $(X,T)$ is said to be \textit{zero-dimensional} if $X$ has a (topological) basis consisting of clopen sets. Suppose that $X$ is a zero-dimensional $G$ system. Then for any $\epsilon >0$, the space $X$ has a finite clopen partition $\mathcal{P}$ with $\diam(\mathcal{P})<\epsilon$. 
For $x \in X$, let $\mathcal{P}(x)$ be the element of $\mathcal{P}$ containing $x$. Then define a map $\pi_{\mathcal{P}} : X \to \mathcal{P}^G$ by the rule $\pi_{\mathcal{P}}(x)(g) = \mathcal{P}(T^g(x))$, and note that $\pi$ is a factor map from $X$ onto $X_{\mathcal{P}} := \pi_{\mathcal{P}}(X)$. 
In fact, by choosing $\epsilon > 0$ sufficiently small, one may ensure that the entropy of the subshift $X_{\mathcal{P}}$ is arbitrarily close to the entropy of  $(X,T)$. 

\subsection{Inverse limits}\label{invsec} 


\begin{defn}
Suppose $\{(Z_n,R_n)\}_{n=0}^{\infty}$ is a sequence of $G$-systems, and for each $n \geq 1$, we have a factor map $\pi_n : Z_n \to Z_{n-1}$. The inverse limit system $(Z,R) = \varprojlim (Z_n,R_n)$ is defined as follows: 
\begin{equation*}
Z = \left\{ (z_0,z_1,z_2,\dots) \in \prod_{n=0}^{\infty} Z_n : \forall n \geq 1, \pi_n(z_n) = z_{n-1} \right\},
\end{equation*}
and  $R^g( z_0,z_1,z_2,\dots) = (R_0^g(z_0), R_1^g(z_1), R_2^g(z_2), \dots)$. 
\end{defn}

We remark that since each $(Z_n,R_n)$ is a $G$-system, so is the inverse limit system $(Z,R)$.
%
Furthermore, if each $Z_n$ is zero-dimensional, then $Z$ is zero-dimensional as well.
We require the following lemma about entropy of inverse limits; for a proof in the case $G = \mathbb{Z}$, see \cite[Lemma 4.9]{linden}. The straight-forward adaptation of this proof for general countable amenable groups is left to the reader.

\begin{lemma} \label{Lemma:InvLimitEntropy}
For any inverse limit system $(Z,R) = \varprojlim (Z_n,R_n)$, we have
\begin{equation*}
h(R) = \lim_n h(R_n).
\end{equation*}
\end{lemma}


Finally, we will need the following simple lemma, which shows that factor maps can be carried to the inverse limit; the proof is standard and left to the reader.

\begin{lemma} \label{Lemma:InvLimitFactor}
Suppose that $X$ is a $G$-system with a surjective factor map $\varphi_0 : X \to Z_0$ and that for each $n \geq 1$, there exist surjective factor maps $\varphi_n : X \to Z_n$ and $\pi_n : Z_n \to Z_{n-1}$ such that $\varphi_{n-1} = \pi_n \circ \varphi_n$. 
Define the map $\varphi : X \to Z$ by the rule $\varphi(x) = (\varphi_0(x),\varphi_1(x),\dots)$. Then $\varphi$ is a surjective factor map.
\end{lemma}

\subsection{Periodic patterns}

\begin{defn}
For any finite alphabet $\mathcal{A}$, countable amenable $G$, $S, T \subset G$ with $S$ finite, $k \in \N$, and $w \in \mathcal{A}^T$, we say that $w$  \textit{has $k$ periods from $S$} if there exists a collection $\{s_1,\dots,s_k\}$ of $k$ distinct elements of $S$ such that $w_g = w_{gs_i}$ whenever $g$ and $gs_i$ are both in $T$. 
In this case, we may refer to $\{s_1,\dots,s_k\}$ as a \textit{period set} for $w$.
\end{defn}

\begin{lemma} \label{Lemma:PeriodicPatterns}
Let $k \in \N$, and let $S \subset G$ be a finite set with $|S| \geq k$. Let $T \subset G$ be any set such that 
\begin{equation*}
k \log_{|\mathcal{A}|} |S| < \frac{|T|}{2k}
\end{equation*}
and for each $s \in S$,
\begin{equation*}
|T \triangle Ts| < \frac{ |T| }{2k^2}. 
\end{equation*}
Then
\begin{equation*}
\Bigl| \bigl\{ w \in \mathcal{A}^{T} : w \text{ has $k$ periods from $S$} \bigr\} \Bigr| \leq |\mathcal{A}|^{2 |T| / k}.
\end{equation*}
\end{lemma}
\begin{proof}
Let $k \in \N$, and let $S \subset G$ be a finite set with $|S| \geq k$. Let $T$ be as above. To simplify the notation in this proof, let $\delta >0$ be such that $\delta < k^{-1}/2$ and such that
\begin{equation*}
k \log_{|\mathcal{A}|} |S|  < \delta |T|,
\end{equation*}
and for each $s \in S$,
\begin{equation*}
|T \triangle Ts| < \frac{ \delta |T| }{k}. 
\end{equation*}

Now let $P \subset S$ such that $|P| = k$. First, define a finite undirected graph with vertex set $V = T$ and edge set $E \subset T \times T$, where $(g,h) \in E$ if there exists $p \in P$ such that $h = gp$ or $g = hp$.
Let $C \subset T$ be the set of vertices corresponding to an arbitrary connected component of $(V,E)$ such that $|C| < k$ (which might not exist). Let $g \in C$. If $p \in P$ and $h = gp^{-1} \in T$, then $g = hp$, which means $(g,h) \in E$, and then $gp^{-1} = h \in C$. Since $|C| < k = |g P^{-1}|$, there exists $p \in P$ such that $gp^{-1} \notin T$, i.e., $g \in T \setminus Tp$. We now conclude that for any connected component $C$ of $(V,E)$ with $|C| < k$, we have
\begin{equation*}
C \subset \bigcup_{p \in P} \bigl(T \setminus Tp \bigr).
\end{equation*}
For $g \in T$, let $C(g)$ denote the connected component of $(V,E)$ containing $g$. Then the above containment and our second hypothesis on $T$ combine to give
\begin{align*}
\Bigl| \bigl\{ g \in T :  |C(g)| < k  \bigr\} \Bigr|  \leq \sum_{p \in P} |T \setminus Tp|  < \frac{\delta |T|}{k} |P| =  \delta |T|.
\end{align*}
Let $N_{\ell}$ be the number of connected components of $(V,E)$ with cardinality $\ell$.
We have
\begin{align*}
|T| & \geq \Bigl| \bigl\{ g \in T : |C(g)| \geq k \bigr\} \Bigr| \\
& = \sum_{\ell = k}^{|T|} N_{\ell} \cdot \ell \\
& \geq k \sum_{\ell = k} N_{\ell},
\end{align*}
and therefore the number of connected components of cardinality at least $k$ is bounded above by $|T|/k$. 

Now we turn to counting patterns in $\mathcal{A}^T$ with period set $P$.
Notice that if $w \in \mathcal{A}^T$ has $P$ as a period set, then $w$ must be constant on any connected component of the graph $(V,E)$. Then the estimates in the previous paragraph yield that
\begin{align*}
\Bigl| \bigl\{ w \in \mathcal{A}^T : P \text{ is a period set for } w \bigr\} \Bigr| & \leq |\mathcal{A}|^{|T|/k} \cdot |\mathcal{A}|^{\delta |T|} = |\mathcal{A}|^{(k^{-1}+\delta) |T|}.
\end{align*}

Finally, let us estimate the number of patterns in $\mathcal{A}^T$ with $k$ periods from $S$. Using the previous estimates and our first hypothesis on $T$, we have
\begin{align*}
\Bigl| \bigl\{ w \in \mathcal{A}^T : & \,  w \text{ has $k$ periods from }  S \bigr\} \Bigr| \\
& \leq \sum_{\substack{P \subset S \\ |P| = k}} \Bigl| \bigl\{ w \in \mathcal{A}^T : P \text{ is a period set for } w \bigr\} \Bigr|  \\
& \leq |\mathcal{A}|^{(k^{-1}+\delta) |T|} \cdot  \bigl| \bigl\{ P \subset S : |P| = k \bigr\} \bigr| \\
& \leq  |\mathcal{A}|^{(k^{-1}+\delta) |T|}  \cdot |S|^k \\
& \leq  |\mathcal{A}|^{(k^{-1}+\delta) |T|}  \cdot |\mathcal{A}|^{\delta |T|} \\
& =  |\mathcal{A}|^{(k^{-1}+2\delta) |T|}. 
\end{align*}
Since $\delta < k^{-1}/2$, the proof of the lemma is complete. 
\end{proof}

\subsection{Other preliminaries}

We denote the usual binary entropy function by $H : [0,1] \to \R$, where $H(x) = - x \log( x) - (1-x) \log(1-x)$, with the convention that $0 \cdot \log 0 = 0$. 
The following fact is elementary and presented without proof. 
\begin{lemma} \label{Lemma:Grinch}
For any $n$ and $\alpha < 1/2$, we have
\begin{equation*}
\sum_{k=0}^{\lfloor \alpha n \rfloor} \binom{n}{k} \leq 2^{H(\alpha) n}.
\end{equation*}
\end{lemma}


\section{Marker Lemma}\label{marker}

Here we will prove a Marker Lemma for countable groups which generalizes the classical Krieger Marker Lemma (\cite{Boyle1983}, \cite{Krieger1982}).
First, we require a definition.

\begin{defn}
Let $\mathcal{F}$ be a finite collection of sets. For $k \in \N$, we say that $\mathcal{F}$ is $k$-\textit{fold disjoint} if for all collections $\{F_1,\dots,F_k \}$ of $k$ distinct elements of $\mathcal{F}$, we have
\begin{equation*}
\bigcap_{\ell = 1}^k F_{\ell} = \varnothing.
\end{equation*}
\end{defn}

Now we state our general Marker Lemma.  Note that it does not require amenability of the group $G$.
\begin{lemma}\label{markerlemma}
Let $G$ be a countable group, and let $X$ be a $G$-subshift.
Let $S \subset G$ and $T \subset G$ be finite, and let $1 \leq k \leq |S|$. Then there exists a clopen set $F \subset X$ such that
\begin{enumerate}
\item the collection $\{ \sigma^s(F) : s \in S \}$ is $(k+1)$-fold disjoint, and
\item if 
\begin{equation*}
x \notin \bigcup_{s \in S^{-1} S} \sigma^s(F),
\end{equation*}
then $x(T)$ has $k$ periods from $S^{-1} S$.
\end{enumerate}
\end{lemma}
The classical Krieger Marker Lemma corresponds to the case $G = \mathbb{Z}$, $k = 1$, $S = [0,n]$, and $T = [-n, n]$ for some $n \in \mathbb{N}$. Our use for arbitrary $k$ can be thought of as allowing shifts of the marker set to have weaker disjointness properties in exchange for stronger periodicity properties away from it.
\begin{proof}
Let $\mathcal{N} = \{w_1,\dots,w_r\}$ be an enumeration of the patterns $w$ in $\mathcal{A}^T$ such that $w$ \textit{does not} have $k$ periods from $S^{-1} S$. 
Inductively define the following sets. Let $A_1 = [w_1]$. For $i = 1,\dots,r-1$, let
\begin{equation*}
A_{i+1} = [w_{i+1}] \setminus \Biggl( \bigcup_{j=1}^i \bigcup_{s \in S^{-1} S} \sigma^s (A_j) \Biggr).
\end{equation*}
Let $F = \bigcup_{i=1}^r A_i$. Note that $F$ is clopen.

To establish (1), suppose for contradiction that there exists $P = \{p_1,\dots,p_{k+1}\} \subset S$ with $|P| = k+1$ and there exists $x \in X$ with
\begin{equation*}
x \in \bigcap_{p \in P} \sigma^p(F). 
\end{equation*}
Then for each $p \in P$, we have that $\sigma^{p^{-1}}(x) \in F = \bigcup_i A_i$. For each $p \in P$, choose $i(p) \in \{1,\dots,r\}$ such that $\sigma^{p^{-1}}(x) \in A_{i(p)}$. We claim that there exists $w \in \mathcal{N}$ such that $\sigma^{p^{-1}}(x) \in [w]$ for all $p \in P$. To see this, suppose for contradiction that there exist $p, q \in P$ such that $i(p) \neq i(q)$. Assume without loss of generality that $i(p) < i(q)$ (otherwise reverse the roles of $p$ and $q$). Then
\begin{equation*}
\sigma^{q^{-1}}(x) \in A_{i(q)} \cap \bigl( \sigma^{q^{-1} p} (A_{i(p)}) \bigr),
\end{equation*}
which gives a contradiction, since $q^{-1}p \in S^{-1} S$ and $A_{i(q)}$ is defined to be disjoint from $\sigma^{s}(A_{i(p)})$ for all $s \in S^{-1} S$. 
Hence, there exists $w \in \mathcal{N}$ such that $\sigma^{p^{-1}}(x) \in [w]$ for all $p \in P$. Therefore
\begin{equation*}
\sigma^{p_1^{-1}}(x) \in [w] \cap \sigma^{p_1^{-1} p_2} [w] \cap \dots \cap \sigma^{p_1^{-1} p_{k+1}}[w].
\end{equation*}
Since $|P| = k+1$, the non-emptiness of the intersection in the previous display gives that the set $\{p_1^{-1} p_2,\dots,p_1^{-1} p_{k+1} \}$ is a set of $k$ periods from $S^{-1}S$ for $w$. However, this contradicts the fact that $w \in \mathcal{N}$. Thus we have established (1). 

Now let $x \in X$ such that $x(T)$ \textit{does not} have $k$ periods from $S^{-1} S$. Then $x(T) = w_i$ for some $i = 1, \dots, r$. If $i = 1$, then $x \in F$. If $i >1$, then either $x \in A_i \subset F$ or else there exists $j < i$ and $s \in S^{-1}S$ such that $x \in \sigma^s(A_j)$. In all cases, we obtain that
\begin{equation*}
x \in \bigcup_{s \in S^{-1} S} \sigma^s(F).
\end{equation*}
Taking the contrapositive, we conclude that if
\begin{equation*}
x \notin \bigcup_{s \in S^{-1} S} \sigma^s(F),
\end{equation*}
then $x(T)$ has $k$ periods from $S^{-1}S$.
This establishes (2) and finishes the proof.
\end{proof}

\section{Density}\label{density}

In this section, we define some basic notions of (upper) density for subsets of countable amenable $G$ (in terms of our previously chosen F\o lner sequence $F_n$). This will be used to quantify a way in which 
visits to the marker set from Lemma~\ref{markerlemma} are rare when $k$ is taken much smaller than $|S|$.

\begin{defn}
Let $X$ be a $G$-subshift, and let $F \subset X$. For a finite set $E \subset G$ and $x \in X$, let
\begin{equation*}
N_{E}(x,F) = \bigl| \{ g \in E : \sigma^g(x) \in F \} \bigr|.
\end{equation*}
Then let
\begin{equation*}
D_n(F) = \sup_{x \in X} \frac{N_{F_n}(x,F)}{|F_n|},
\end{equation*}
and $\overline{D}(F) = \limsup_n D_n(F)$.
\end{defn}


\begin{defn}
Let $X$ be a $G$-subshift. Let $S \subset G$ be finite, and let $k \in \N$. We say that $F \subset X$ is \textit{$(S,k)$-disjoint} if $\{ \sigma^{s}(F) : s \in S\}$ is $k$-fold disjoint. 
\end{defn}

\begin{lemma} \label{Lemma:Density}
Let $X$ be a $G$-subshift. Let $\{S_1,\dots,S_N\}$ be a collection that $\delta$-quasitiles $G$ with $m = \min_i |S_i|$ and $S = \bigcup_i S_i = S_N$, and let $k \geq 1$. If $F \subset X$ is $(S^{-1},k)$-disjoint, then 
\begin{equation*}
\overline{D}(F) \leq \frac{k (1+\delta)}{(1-\delta)m} + \delta.
\end{equation*}
\end{lemma}
\begin{proof}
By Lemma \ref{Lemma:Coverings}, for all large enough $n$, there exists $C = C(S,n,\delta)$ such that $\{Sc : c \in C\}$ $(1-\delta)$-covers $F_n$ and
\begin{equation*}
|C| \leq \frac{(1+\delta) |F_n|}{(1-\delta) m}.
\end{equation*}
Also, since $F$ is $(S^{-1},k)$-disjoint, for any $x \in X$, we have
\begin{align*}
N_{S}(x,F) & = | \{ g \in S : \sigma^g(x) \in F\} | \\
& = | \{ g \in S : x \in \sigma^{g^{-1}}(F) \} | \\
& \leq k.
\end{align*}
Then for any $x \in X$, the previous two displays and the fact that $\{Sc : c \in C\}$ $(1-\delta)$-covers $F_n$ gives that
\begin{align*}
N_{F_n}(x,F) & \leq N_{SC}(x,F) + N_{F_n \setminus SC}(x,F) \\
& \leq \sum_{c \in C} N_{Sc}(x,F) + |F_n \setminus SC| \\
& \leq \sum_{c \in C} N_{S}(\sigma^c(x),F) + \delta |F_n| \\
&\leq k |C| + \delta |F_n| \\
& \leq \frac{k (1+\delta) |F_n|}{(1-\delta)m} + \delta |F_n|.
\end{align*}
After dividing by $|F_n|$, taking the supremum over $x \in X$, and letting $n$ tend to infinity, we obtain the desired estimate.
\end{proof}

Now we show that if a factor map only changes a small percentage of symbols, then the topological conditional entropy is small, which implies by Theorem~\ref{bowenfiber} that the entropy drop over the factor is also small.
\begin{defn}
Let $\pi : X \to Y$ be a factor map between $G$-subshifts. For a finite set $E \subset G$, $y \in Y$, and $x \in \pi^{-1}(y)$, we define
\begin{equation*}
N_E(x,y) = |\{ g \in E : x_g \neq y_g \}|.
\end{equation*}
Then define
\begin{equation*}
D_n(y) = \sup_{x \in \pi^{-1}(y)} \frac{N_{F_n}(x,y)}{|F_n|},
\end{equation*}
and $\overline{D}(\pi) = \limsup_n \sup_{y \in Y} D_n(y)$. 
\end{defn}

\begin{lemma} \label{Lemma:SmallGap}
Let $\pi : X \to Y$ be a factor map between $G$-subshifts $(X,T)$ and $(Y,S)$ on alphabet $\mathcal{A}$. Suppose that $ \overline{D}(\pi) < 1/2$. Then $h(T \mid S) \leq H(\overline{D}(\pi)) + \overline{D}(\pi) \log|\mathcal{A}|$.
\end{lemma}
\begin{proof}
Let $\gamma = \overline{D}(\pi)$, and let $\epsilon >0$ be such that $\gamma + \epsilon < 1/2$. 
Choose $n$ large enough so that 
for all $y \in Y$ and $x \in \pi^{-1}(y)$, we have
\begin{equation*}
 \frac{N_{F_n}(x,y)}{|F_n|} < \gamma + \epsilon.
\end{equation*}

Fix any $y \in Y$. For any $n$, define $\mathcal{L}_{F_n}(\pi^{-1}(y)) = \{x(F_n) \ : \ x \in \pi^{-1}(y)\}$. For any $x \in \pi^{-1}(y)$, define 
$K_{n,x} = \{ g \in F_n : x_g \neq y_g \}$, and note that
\begin{equation*}
|K_{n,x}| = N_{F_n}(x,y) \leq (\gamma+\epsilon) |F_n|.
\end{equation*}

Since $y$ is fixed, we see that $x(F_n)$ may be determined by $K_{n,x}$ (off of which $x$ is equal to $y$) and $x(K_{n,x})$, and therefore
\[
|\mathcal{L}_{F_n}(\pi^{-1}(y))| \leq 2^{H(\gamma+\epsilon) |F_n|} \cdot |\mathcal{A}|^{(\gamma+\epsilon)|F_n|},
\]
where we have used Lemma \ref{Lemma:Grinch}.

Taking logarithms, dividing by $|F_n|$, and letting $n$ tend to infinity gives
\begin{equation*}
h( T \mid S) \leq H(\gamma+\epsilon) + (\gamma+\epsilon)\log|\mathcal{A}|.
\end{equation*}
Since $\epsilon$ may be taken arbitrarily small, we obtain the desired result.
\end{proof}

The following corollary is immediate via Theorem~\ref{bowenfiber}.

\begin{cor}\label{Cor:SmallGap}
Under the assumptions of Lemma~\ref{Lemma:SmallGap}, $h(T) - h(S) \leq H(\overline{D}(\pi)) + \overline{D}(\pi) \log|\mathcal{A}|$.
\end{cor}

\section{Proof of Theorem \ref{Thm:SubshiftEntropies}}\label{mainprf}

In this section we present the proof of Theorem \ref{Thm:SubshiftEntropies}. The main ingredient in the proof is Proposition \ref{Prop:Rainbow}. We present this proposition first and then apply it in the proof of the main theorem below.


\begin{prop} \label{Prop:Rainbow}
Let $(X,R)$ be a subshift. Then for any $\epsilon \in (0,1/2)$, there exists a finite sequence of subshifts $(X_0,R_0),\dots,(X_{M+1},R_{M+1})$ and factor maps $\varphi_{m+1} : X_m \to X_{m+1}$ for $0 \leq m \leq M$ such that $(X_0,R_0) = (X,R)$, $(X_{M+1},R_{M+1})$ is the trivial (one-point) system, and $h(R_m \mid R_{m+1}) < \epsilon$ for all $0 \leq m \leq M$.
\end{prop}
\begin{proof}
Let $(X,R)$ be a $G$-subshift with alphabet $\mathcal{A}$. We assume that $|\mathcal{A}| \geq 2$, since when $|\mathcal{A}| = 1$, we have $h(R) = 0$ and the proposition holds trivially.



Let $\epsilon \in (0,1/2)$. 
Choose $k \geq 1$ such that $4 \log( |\mathcal{A}| ) / k < \epsilon$. 
Choose $\delta \in (0,1)$ and $m_0 \geq 1$ such that
\begin{equation} \label{Eqn:Kacey}
H \left( \frac{(k+1)(1+\delta)}{(1-\delta) m_0} + \delta \right) + \left( \frac{(k+1)(1+\delta)}{(1-\delta) m_0} + \delta \right) \log |\mathcal{A}| < \epsilon/2.
\end{equation}
Choose a collection $\{S_1,\dots,S_M\}$ that $\delta$-quasitiles $G$ such that $\min_i |S_i| \geq m_0$ and $S = \bigcup_iS_i$ has cardinality at least $k$.
Choose $\eta \in (0,1)$ such that 
\begin{equation}\label{Eqn:Humbug}
\log(|\mathcal{A}|) (2k^{-1} (1- \eta)^{-1}  + \eta) < \epsilon/2.
\end{equation}
Choose a collection $\{T_1,\dots,T_N\}$ that $\eta$-quasitiles $G$ such that \newline $2k^2 \log_{|\mathcal{A}|} |SS^{-1}| \leq |T_1| = \min_i |T_i|$ and for all $s \in SS^{-1}$ and $i = 1,\dots,N$, we have
\begin{equation*}
|T_i \, \triangle \, T_i s| < \frac{|T_i|}{2 k^2}.
\end{equation*}
Let $T = \bigcup_i T_i$.

Now apply the Marker Lemma (Lemma \ref{markerlemma}) with parameters $k,S^{-1}$, and $T$. We get a clopen set $F \subset X$ such that $F$ is $(S^{-1},k+1)$-disjoint and if $x \in X$ satisfies
\begin{equation*}
x \notin \bigcup_{s \in S S^{-1}} \sigma^s(F),
\end{equation*}
then $x(T)$ has $k$ periods in $S S^{-1}$. 
By Lemma \ref{Lemma:Density}, we obtain that
\begin{equation} \label{Eqn:Brandi}
\overline{D}(F) \leq \frac{(k+1)(1+\delta)}{(1-\delta) m_0} + \delta.
\end{equation}

Before defining our factor maps, we require a few more definitions. Let $G = \{g_k\}_{k=1}^{\infty}$ be an enumeration of $G$, with the convention that $g_1 =e$. Let $G_m = \{ g_1,\dots,g_m\}$. We suppose that $a$ and $b$ are symbols that are \textit{not} contained in $\mathcal{A}$, and we let $\mathcal{B} = \mathcal{A} \cup \{a,b\}$.

Now we define our factor maps. First, let $\varphi_1 : X \to \mathcal{B}^G$ be defined by the rule
\begin{equation*}
\varphi_1(x)_g = \left\{ \begin{array}{ll}
                                          a, & \sigma^{g}(x) \in F \\
                                          x_g, & \text{otherwise}.
                                    \end{array}
                             \right.
\end{equation*}
Since $F$ is clopen, $\varphi_0$ is a sliding block code. 
Let $X_1 = \varphi_1(X)$.
Inductively, suppose that $\varphi_1,\dots,\varphi_m$ and $X_1,\dots,X_m$ have been defined. Define $\varphi_{m+1} : X_m \to \mathcal{B}^G$ by the rule
\begin{equation*}
\varphi_{m+1}(x)_g = \left\{ \begin{array}{ll}
                                          b, & \text{ if } x_g \neq a \text{ and } x_{g_{m+1}^{-1} g} = a \\
                                          x_g, & \text{otherwise}.
                                    \end{array}
                             \right.
\end{equation*}
It is clear that $\varphi_{m+1}$ is a sliding block code.
Let $X_{m+1} = \varphi_{m+1}(X_m)$. 
This concludes our definition of the factor maps $\{\varphi_m\}_{m=0}^{\infty}$ and the subshifts $\{(X_m,R_m)\}_{m=0}^{\infty}$.

The remainder of the proof will be devoted to showing that the set of entropies of the subshifts $(X_m,R_m)$ is $\epsilon$-dense in the 
interval $[0, h(R)]$ by verifying that the topological conditional entropies $h(R_m \mid R_{m+1})$ are smaller than $\epsilon$
and that the entropy $h(R_m)$ is `eventually small,' i.e. $h(R_m) < \epsilon$ for sufficiently large $m$.

Both claims will be proved by appealing to properties of $F$ guaranteed by the Marker Lemma. The former will follow from the fact that visits to $F$ have small density, meaning that the changes
made via each $\varphi_m$ have small density. The latter will follow from the fact that portions of points of $X$ which are not near visits to $F$ are highly periodic, and since letters at locations near visits to $F$ are changed to $a$ for large $m$, such $(X_m,R_m)$ will have small entropy by Lemma~\ref{Lemma:PeriodicPatterns}.

Now we will establish that each topological conditional entropy satisfies $h(R_m \mid R_{m+1}) < \epsilon$.


\begin{claim}\label{easyclaim0}
 $\overline{D}(\varphi_1) \leq \overline{D}(F)$. 
\end{claim}
(In fact, this inequality is an equality, but we will not need that fact.)
\begin{proof}
Let $\epsilon_1 >0$. Choose $n$ large enough so that $D_n(F) \leq \overline{D}(F)+\epsilon_1$.
Let $y \in X_1$ and $x \in \varphi_1^{-1}(y)$. Then 
\begin{align*}
N_{F_n}(x,y) & = | \{ g \in F_n : x_g \neq y_g \} | \\
& = | \{ g \in F_n : \sigma^g(x) \in F \} | \\
& = N_{F_n}(x,F) \\
& \leq (\overline{D}(F)+ \epsilon_1)|F_n|
\end{align*}
Dividing by $|F_n|$, taking supremum over $y \in Y$ and $x \in \varphi_1^{-1}(y)$, and letting $n$ tend to infinity yields
\begin{equation*}
\overline{D}(\varphi_1) \leq \overline{D}(F)+\epsilon_1.
\end{equation*}
Since $\epsilon_1$ may be taken arbitrarily small, we obtain that $\overline{D}(\varphi_1) \leq \overline{D}(F)$. 
\end{proof}
\noindent

%
Note that $\overline{D}(F) < 1/2$ (by (\ref{Eqn:Kacey}) and (\ref{Eqn:Brandi})). 
Then by Corollary \ref{Cor:SmallGap} and Claim~\ref{easyclaim0}, we see that 
\begin{align*}
h(R \mid R_1) &  \leq  H(\overline{D}(\varphi_1)) + \overline{D}(\varphi_1) \log |\mathcal{A}| \\ 
& \leq H(\overline{D}(F)) + \overline{D}(F) \log |\mathcal{A}|. 
\end{align*}
Then by (\ref{Eqn:Kacey}) and (\ref{Eqn:Brandi}), we conclude that $h(R \mid R_1) < \epsilon$.

\begin{claim}
For all $m \geq 1$, we have $\overline{D}(\varphi_{m+1}) \leq \overline{D}(F)$.
\end{claim}
\begin{proof}
 Let $\epsilon_1 >0$. Choose $n$ large enough so that $D_n(F) \leq \overline{D}(F) +\epsilon_1$ and $|(g_{m+1}^{-1} F_n) \setminus F_n| \leq \epsilon_1 |F_n|$. Let $y \in X_{m+1}$ and $x \in \varphi_{m+1}^{-1}(y)$. Let $z \in X$ be such that $x =  \varphi_m \circ \dots \varphi_1(z)$. Then
\begin{align*}
N_{F_n}(x,y)  &= | \{ g \in F_n : x_g \neq y_g \}| \\
& \leq | \{ g \in F_n : x_g \neq a \text{ and } x_{g_{m+1}^{-1} g} = a| \\
& \leq | \{ g \in F_n : x_{g_{m+1}^{-1} g} = a \} | \\
& = N_{g_{m+1}^{-1} F_n}(z,F) \\
& \leq N_{F_n}(z,F) + N_{(g_{m+1}^{-1} F_n) \setminus F_n}(z,F) \\
& \leq (\overline{D}(F)+\epsilon_1)|F_n| + |(g_{m+1}^{-1} F_n) \setminus F_n| \\
& \leq (\overline{D}(F)+\epsilon_1)|F_n| + \epsilon_1 |F_n|.
\end{align*}
Dividing by $|F_n|$, taking the supremum over $y \in X_{m+1}$ and $x \in \varphi_{m+1}^{-1}(y)$, and letting $n$ tend to infinity yields
\begin{equation*}
\overline{D}(\varphi_{m+1}) \leq \overline{D}(F)+ 2\epsilon_1.
\end{equation*}
Since $\epsilon_1$ may be taken arbitrarily small, we obtain that $\overline{D}(\varphi_{m+1}) \leq \overline{D}(F)$. 
\end{proof}
%
By Corollary \ref{Cor:SmallGap} and the previous claim, we see that 
\begin{align*}
h(R_m \mid R_{m+1}) & \leq H(\overline{D}(\varphi_{m+1})) + \overline{D}(\varphi_{m+1}) \log |\mathcal{A}| \\
& \leq H(\overline{D}(F)) + \overline{D}(F) \log |\mathcal{A}| .
\end{align*}
Then by (\ref{Eqn:Kacey}) and (\ref{Eqn:Brandi}), we conclude that $h(R_m \mid R_{m+1}) < \epsilon$.

Now we proceed to show that $h(R_m) < \epsilon$ for sufficiently large $m$.
Note that by our choice of the quasitiles $T_1,\dots,T_N$, we may apply Lemma \ref{Lemma:PeriodicPatterns} with parameters $k$, $SS^{-1}$, and $T_i$, obtaining that for each $i$, we have
\begin{equation} \label{Eqn:Loretta}
| \{ v \in \mathcal{A}^{T_i} : v \text{ has $k$ periods from $SS^{-1}$} \} | \leq |\mathcal{A}|^{2|T_i|/k}.
\end{equation}
\begin{lemma} \label{Lemma:SmallEntropy}
For large enough $m$, we have $h(R_m) < \epsilon$.
\end{lemma}
\begin{proof}
Choose $m$ large enough that $T S S^{-1} \subset G_m$. Let $\delta_1 \in ( \overline{D}(F), 1/2)$ and $\delta_2>0$ be arbitrary. Choose $n$ large enough that $D_n(F) \leq \delta_1$ and 
\begin{equation*}
\max \left( \frac{ |F_n \, \triangle \, G_m^{-1} F_n|}{|F_n|}, \frac{|F_n \, \triangle TT^{-1} F_n|}{|F_n|} \right) < \delta_2.
\end{equation*}
Since $\{T_1,\dots,T_N\}$ is a set of $\eta$-quasitiles, there exists a collection $\{C_1,\dots,C_N\}$ of center sets corresponding to $F_n$ with the additional property that if $c \in C_i$, then $T_i c \cap F_n \neq \varnothing$.

Let $w \in \mathcal{L}_{F_n}(X_m)$. Choose $y \in X_m$ such that $y(F_n) = w$, and choose $x \in X$ such that $y = \varphi_{m} \circ \dots \varphi_1(x)$. Let $J_w = \{ g \in G_m^{-1} F_n : y_g = a\}$. Note that for $g \in F_n$, we have that $w_g = a$ if and only if $g \in J_w$, and $w_g = b$ if and only if $g \in (G_m J_w) \setminus J_w$.  Furthermore,
\begin{align} \label{Eqn:Armando}  \begin{split}
|J_w| &= |J_w \cap F_n| + |J_w \setminus F_n| \\
& \leq N_{F_n}(x,F) + | (G^{-1}_m F_n) \setminus F_n| \\
& \leq \delta_1 |F_n| + \delta_2 |F_n|,
\end{split}
\end{align}
where we have used our choice of $n$ in the last estimate. 

Now for each $i$, let $C_i^w$ be the set of $c \in C_i$ such that $(F_n \cap T_i c) \setminus (G_m J_w) \neq \varnothing$. Note that $C_i^w$ is completely determined by $J_w$ (along with the already chosen $F_n$, $T_i$, and $G_m$).  

\begin{claim} \label{Claim:Handel}
If $c \in C_i^w$, then $x(T_i c)$ has $k$ periods from $SS^{-1}$. 
\end{claim}
\begin{proof}
To begin, suppose that $c \in C_i$, $g \in F_n \cap T_ic$, and
\begin{equation*}
\sigma^c(x) \in \bigcup_{s \in SS^{-1}} \sigma^s (F).
\end{equation*}
Then there exists $s \in SS^{-1}$ such that $\sigma^{sc}(x) \in F$. Hence $y_{sc}=a$. Let $g' = sc$. Then $c = s^{-1} g'$. Now let $t \in T_i$ be such that $tc =g$. Then $g' = s t^{-1} g \in SS^{-1} T^{-1} F_n \subset G_m^{-1} F_n$, by our choice of $m$. Therefore $g' \in J_w$, and then $g = t s^{-1} g' \in T SS^{-1} J_w \subset G_m J_w$, again using our choice of $m$.
We conclude that if $c \in C_i$ and
\begin{equation*}
\sigma^c(x) \in \bigcup_{s \in SS^{-1}} \sigma^s (F),
\end{equation*}
then $F_n \cap T_i c \subset G_m J_w$. By the contrapositive, if $c \in C_i$ and $(F_n \cap T_i c) \setminus (G_m J_w) \neq \varnothing$, then 
\begin{equation*}
\sigma^c(x) \notin \bigcup_{s \in SS^{-1}} \sigma^s (F),
\end{equation*}
which gives that $x(Tc)$ has $k$ periods from $SS^{-1}$ by our choice of $F$.
Thus, we have shown that if $c \in C_i^w$, then $x(Tc)$ has $k$ periods from $SS^{-1}$, and therefore so does $x(T_ic)$ (since $T_i \subset T$).
\end{proof}


For a finite set $E$, let $\mathcal{P}(E)$ denote the power set of $E$.
Now consider the map $\phi : \mathcal{L}_{F_n}(X_m) \to \mathcal{P}(G_m^{-1} F_n)$  defined by $w \mapsto J_w$. 
\begin{claim} \label{Claim:Hearts}
For each $J \subset \mathcal{P}(G_m^{-1} F_n)$, we have
\begin{equation*}
|\phi^{-1}(J)| \leq |\mathcal{A}|^{\eta |F_n|} \cdot |\mathcal{A}|^{(2/k) \sum_i |T_i| \cdot |C_i|}.
\end{equation*}
\end{claim}
\begin{proof}
Let $J \in \mathcal{P}(G_m^{-1} F_n)$. Define $C'_i = \{ c \in C_i : (F_n \cap T_i c) \setminus (G_m J) \neq \varnothing \}$. Now let $w \in \phi^{-1}(J)$, i.e., $J_w = J$, and let $x \in X$ be such that $x(F_n) = w$.
Note that since $J_w = J$, we also have $C_i^w = C'_i$ for each $i$. 

Let $g \in F_n$. For $g \in G_mJ$, we have that $w_g = a$ whenever $g \in J$ and $w_g = b$ whenever $g \notin J$. Now suppose $g \in T_i c \setminus (G_mJ)$ for some $c \in C_i$. Then $w_g = x_g$. Also, we have $c \in C'_i$, and by Claim \ref{Claim:Handel}, $x(T_ic)$ has $k$ periods from $SS^{-1}$. Hence $w \in \phi^{-1}(J)$ is uniquely determined by a tuple of the form
\begin{equation*}
\left( \bigl(x(T_1c) \bigr)_{c \in C'_1}, \ \ldots, \ \bigl(x(T_Nc)\bigr)_{c \in C'_N}, \ w(F_n \setminus ( \bigcup_i T_i C_i)) \right)
\end{equation*}
where each $x(T_ic)$ has $k$ periods from $S S^{-1}$. 
Thus, we have
\begin{align*}
|\phi^{-1}(J)| & \leq  \prod_{i=1}^N | \{ v \in \mathcal{A}^{T_i} : v \text{ has $k$ periods from $SS^{-1}$} \}|^{|C'_i|}  \\
 & \quad  \quad \cdot |\mathcal{A}|^{|F_n \setminus (\bigcup_i T_i C_i)|} \\
 & \leq  \prod_{i=1}^N | \{ v \in \mathcal{A}^{T_i} : v \text{ has $k$ periods from $SS^{-1}$} \}|^{|C_i|} \\
 & \quad  \quad \cdot |\mathcal{A}|^{|F_n \setminus (\bigcup_i T_i C_i)|}.
\end{align*}
Then by Lemma \ref{Lemma:PeriodicPatterns} and the fact that $\{T_1 C_1, \dots, T_N C_N\}$ $(1-\eta)$-covers $|F_n|$, we obtain the desired inequality
\begin{equation*}
|\phi^{-1}(J)| \leq |\mathcal{A}|^{\eta |F_n |} \cdot |\mathcal{A}|^{(2/k) \sum_i |T_i| \cdot |C_i|}.
\end{equation*}
\end{proof}


Finally, using (\ref{Eqn:Armando}) in combination with Lemma \ref{Lemma:Grinch} and Claim \ref{Claim:Hearts} yields the following estimate on the cardinality of $|\mathcal{L}_{F_n}(X_m)|$:
\begin{align*}
|\mathcal{L}_{F_n}(X_m)| 
& \leq 2^{H(\delta_1+\delta_2) |F_n|}  |\mathcal{A}|^{\eta |F_n|} \cdot |\mathcal{A}|^{(2/k) \sum_i |T_i| \cdot |C_i|}.
\end{align*}

Using the $\eta$-disjointness of $T_1,\dots,T_N$, we see that for each $i$, we have
\begin{align*}
|T_i C_i| & = \left| \bigcup_{c \in C_i} T_i c \right| \\
& \geq \sum_{c \in C_i} (1-\eta) |T_i c | \\
& \geq (1-\eta) |T_i| \cdot |C_i|.
\end{align*}
Recall that by our choice of centers, if $c \in C_i$, then $T_i c \cap F_n \neq \varnothing$. Let $g \in T_i c \cap F_n$. Then $g = tc$ for some $t \in T_i$, and so $c = t^{-1} g \in T^{-1} F_n$. Hence $T_i C_i \subset TT^{-1} F_n$.
Combining the previous displayed formula with the quasi-invariance of $F_n$ with respect to $TT^{-1}$ and $\delta_2$ (by choice of $n$), we obtain
\begin{align*}
\sum_i |T_i| \cdot |C_i| & \leq \frac{1}{1-\eta} \sum_i |T_i C_i| \\
& \leq \frac{1}{1-\eta}  \left| \bigcup_i T_i C_i \right| \\
& \leq \frac{1}{1-\eta} | TT^{-1} F_n | \\
& \leq \frac{1+\delta_2}{1-\eta} |F_n|.
\end{align*}
Finally, putting together all of the above estimates, we get
\begin{align*}
|\mathcal{L}_{F_n}(X_m)| & \leq 2^{ H(\delta_1 +\delta_2)|F_n|} |\mathcal{A}|^{(2/k) \sum_i |T_i| \cdot |C_i| } |\mathcal{A}|^{\eta |F_n|} \\
& \leq 2^{ H(  \delta_1 + \delta_2)|F_n|} |\mathcal{A}|^{(2/k)  \frac{1+\delta_2}{1-\eta} |F_n|  } |\mathcal{A}|^{\eta |F_n|}.
\end{align*}
Taking logarithms, dividing by $|F_n|$, and letting $n$ tend to infinity, we get
\begin{equation*}
h(R_m) \leq H(\delta_1+\delta_2) + \left( \frac{2(1+\delta_2)}{k(1-\eta)} + \eta \right) \log |\mathcal{A}|.
\end{equation*}
Since $\delta_2>0$ was arbitrary, and since $\delta_1 \in (\overline{D}(F),1/2)$ was arbitrary, we see that
\begin{equation*}
h(R_m) \leq H(\overline{D}(F)) + \left( \frac{2}{k(1-\eta)} + \eta \right) \log |\mathcal{A}|.
\end{equation*}
By (\ref{Eqn:Kacey}), (\ref{Eqn:Humbug}), and (\ref{Eqn:Brandi}), we conclude that $h(R_m) < \epsilon$. This finishes the proof of Lemma \ref{Lemma:SmallEntropy}.
\end{proof}

Choosing $M$ sufficiently large, we have now established that $h(R_M) < \epsilon$. Now let $(X_{M+1},R_{M+1})$ be the one-point system, which is trivially a factor of $(X_M,R_{M})$ such that $h(R_M \mid R_{M+1}) = h(R_M) < \epsilon$.  As a result, we have finished the proof of the proposition.
\end{proof}

We are now prepared to prove Theorem \ref{Thm:SubshiftEntropies}.

\vspace{2mm}

\begin{PfofThm1}

Let $G$ and $(X,T)$ be as in the statement of the theorem. 
Let $\epsilon >0$.
By Proposition \ref{Prop:Rainbow}, there exist subshifts $(X_0,T_0),\dots,(X_{M+1},T_{M+1})$ and factor maps $\varphi_{m+1} : X_m \to X_{m+1}$ for $0 \leq m \leq M$ such that $(X_0,T_0) = (X,T)$, $(X_{M+1},T_{M+1})$ is the one-point system, and $h(T_m \mid T_{m+1}) < \epsilon$ for each $0 \leq m \leq M$. 
Note that each $(X_m,T_m)$ is a factor of $(X,T)$ (with factor map $\varphi_m \circ \dots \circ \varphi_1$).
Then by Theorem \ref{bowenfiber}, we have that $h(T_m) - h(T_{m+1}) \leq h(T_m \mid T_{m+1}) < \epsilon$ for each $m = 0,\dots, M$. Also, since $(X_{M+1},T_{M+1})$ is the one-point system, we see that $h(T_M) < \epsilon$. Hence we have that $\{h(T_m) : 0 \leq m \leq M \}$ is $\epsilon$-dense in $[0,h(T)]$.
Since $\epsilon >0$ was arbitrary, we conclude that  $\mathcal{H}_{sub}(X)$ is dense in $[0,h(T)]$, which finishes the proof of Theorem \ref{Thm:SubshiftEntropies}.
\end{PfofThm1}



\section{Proof of Theorem \ref{Thm:RelativeSubshiftEntropies}}\label{relprf}



Our first step is to prove an auxiliary result that will form the majority of the proofs of both Theorem \ref{Thm:RelativeSubshiftEntropies} and Theorem \ref{Thm:ZeroDimSystems}. Informally, this result shows that the subshift factors $(X_m,T_m)$ from Proposition~\ref{Prop:Rainbow} still have 
$\epsilon$-dense entropies even after taking a topological joining with another $G$-system. (Recall from Section \ref{Sect:TopJoinings} that $J(\pi,\varphi)$ is our notation for the topological joining of two factor maps $\pi$ and $\varphi$ with common domain.)


\begin{prop} \label{Prop:DoubleRainbow}
Suppose $(X,T)$ is a zero-dimensional $G$-system with finite entropy, $\varphi : X \to Z$ is a factor map onto a subshift $(Z,R)$, and $\pi : X \to Y$  is a factor map onto a $G$-system $(Y,S)$. Let $(W,Q) = J(\varphi,\pi)$. Then for any $\epsilon >0$, there exist subshifts $(Z_0,R_0),\dots,(Z_{M},R_{M})$ and factor maps $\psi_m : X \to Z_m$ for $0 \leq m \leq M+1$ such that if $(W_m,Q_m) = J(\psi_m,\pi)$, then $\{h(Q_m) : 0 \leq m \leq M\}$ is $\epsilon$-dense in $[h(S), h(Q)]$. 
\end{prop}
\begin{proof}
Assume the hypotheses of the proposition. Let $\epsilon >0$. By applying Proposition \ref{Prop:Rainbow} to the subshift $(Z,R)$, we obtain that there exist subshifts $(Z_0,R_0),\dots,(Z_{M+1},R_{M+1})$ and factor maps $\varphi_{m+1} : Z_m \to Z_{m+1}$ for $0 \leq m \leq M$ such that $(Z_0,R_0) = (Z,R)$, $(Z_{M+1},R_{M+1})$ is the one-point system, and $h(R_m \mid R_{m+1}) < \epsilon$ for $0 \leq m \leq M$.   Define the factor maps $\psi_m : X \to Z_m$ by setting $\psi_m = (\varphi_m \circ \dots \circ \varphi_1) \circ \varphi$ for each $1 \leq m \leq M$. Let $(W_m,Q_m)$ be the topological joining $J(\psi_m,\pi)$. It remains to show that $\{h(Q_m) : 0 \leq m \leq M\}$ is $\epsilon$-dense in $[h(S), h(Q)]$. 

First, note that $(W,Q) = (W_0,Q_0)$, and for $0 \leq m \leq M$, 
let $\varphi_{m+1} \otimes \text{id} : W_m \to W_{m+1}$ be the factor map given by $(\varphi_{m+1} \otimes \text{id})(z,y) = (\varphi_{m+1}(z),y)$. Then by Theorem \ref{bowenfiber} and Lemma \ref{Lemma:SpecialFlower}, we have that for each $0 \leq m \leq M$, we have $h(Q_m) - h(Q_{m+1}) \leq h(Q_m \mid Q_{m+1}) \leq h(R_m \mid R_{m+1}) < \epsilon$. 
Finally, since $(Z_{M+1},R_{M+1})$ is the one-point system and $h(R_M \mid R_{M+1}) < \epsilon$, we see that $h(R_M) = h(R_M \mid R_{M+1})$ and then $h(Q_m) \leq h(S) + h(R_M) < h(S) + \epsilon$. Thus, we have established that $\{h(Q_m) : 0 \leq m \leq M\}$ is $\epsilon$-dense in $[h(S), h(Q)]$.
\end{proof}

\begin{rmk}
We note for future reference that $(W_m, Q_m)$ are intermediate factors for $\pi$; indeed, they factor onto $Y$ via the projection $\pi_2$ to the second coordinate, and are factors of $(X,T)$, as they are topological joinings $J(\psi_m,\pi)$ of factors of $(X, T)$.
\end{rmk}


\vspace{2mm}

\begin{PfofThm2}
 
Suppose $(X,T)$ and $(Y,S)$ are subshifts and $\pi : X \to Y$ is a factor map. For any $\epsilon > 0$, if we define $(Z,R) = (X,T)$ and let $\varphi : Z \to X$  be the identity map, then Proposition \ref{Prop:DoubleRainbow} yields intermediate factors $(W_m, Q_m)$ of $\pi$ for $0 \leq m \leq M$ whose entropies are $\epsilon$-dense in $[h(S), h(T)]$. Since each $(W_m, Q_m)$ is a subshift (as it is a topological joining of factors between subshifts) and $\epsilon>0$ was arbitrary, the proof is complete.

\end{PfofThm2}
\section{Proof of Theorem \ref{Thm:ZeroDimSystems}}\label{zeroprf}


Our proof proceeds by first using Proposition~\ref{Prop:DoubleRainbow} to show that the entropies of zero-dimensional intermediate factors are dense in $[h(S), h(T)]$, and then using inverse limits to show that $\mathcal{H}_0^{\pi}(X,Y) = [h(S), h(T)]$.

\vspace{2mm}

\begin{PfofThm3}

Suppose that $G$, $(X,T)$, $(Y,S)$, and $\pi$ are as in the theorem. We begin by showing that the entropies of zero-dimensional intermediate factors are dense in $[h(S), h(T)]$. 

Let $\epsilon >0$. Since $X$ is zero-dimensional, there exists a finite clopen partition $\mathcal{P}$ of $X$ such that $h(T_{\mathcal{P}}) > h(T) - \epsilon/2$. By Proposition \ref{Prop:DoubleRainbow} (taking $(Z,R)$ to be $(X_{\mathcal{P}}, T_{\mathcal{P}})$, $\varphi = \pi_{\mathcal{P}}$, and parameter $\epsilon/2$), there exist intermediate factors $(W_m, Q_m)$ of $\pi$ for $0 \leq m \leq M$ whose entropies are $\epsilon/2$-dense in 
$[h(S), h(T_{\mathcal{P}})]$, and therefore $\epsilon$-dense in $[h(S), h(T)]$. Since each $(W_m, Q_m)$ is zero-dimensional and $\epsilon>0$ was arbitrary, we conclude that $\mathcal{H}_0^{\pi}(X,Y)$ is dense in $[h(S),h(T)]$.


To complete the proof, we will construct zero-dimensional intermediate factors with arbitrary entropy $r \in [h(S), h(T)]$ as inverse limits of 
intermediate zero-dimensional factors using Lemma~\ref{Lemma:InvLimitFactor}. 

The case $r = h(S)$ is trivial, and so we let $r \in (h(S),h(T)]$. Let $(Z_0, R_0) = (Y, S)$,  $\varphi_0 = \pi$, and $\psi_0$ be the identity map on $Y$. 
By the density of $\mathcal{H}_0^{\pi}(X,Y)$ in $[h(S),h(T)]$, there exists a zero-dimensional system $(Z_1, R_1)$ and 
factor maps $\varphi_1 : X \to Z_1$ and $\pi_1 : Z_1 \to Z_0$ such that $\pi = \pi_1 \circ \psi_1$ and $h( R_1) \in (r-1,r)$. Now suppose we have defined $(Z_n, R_n)$, $\varphi_n : X \to Z_n$, and $\pi_n : Z_n \to Z_{n-1}$. By the density of $\mathcal{H}_0^{\varphi_n}(X,Z_n)$ in $[h(R_n),h(T)]$, there exists a zero-dimensional system $(Z_{n+1}, R_{n+1})$ and factor maps $\varphi_{n+1} :  X \to Z_{n+1}$ and $\pi_{n+1} : Z_{n+1} \to 
Z_n$ such that $\pi_n = \pi_{n+1} \circ \varphi_{n+1}$ and $h(R_{n+1}) \in (r - \frac{1}{n+1},r)$. Let $(Z, R) = \varprojlim (Z_n, R_n)$, and let $\varphi : X \to Z$ be the natural factor map (as in Lemma \ref{Lemma:InvLimitFactor}). Then $Z$ is an intermediate factor between $(X,T)$ and $(Y,S)$, and  by Lemma \ref{Lemma:InvLimitEntropy} and our choice of $h(R_n)$ for each $n$, we have that $h(R) = \lim_n h(R_n) = r$, which concludes the proof of the theorem.
\end{PfofThm3}

\bibliographystyle{abbrv}
\bibliography{FactorEntropiesRefs}

\end{document}